\documentclass[11pt,a4paper]{article}
\usepackage{caption}
\usepackage{epsf,epsfig,amsfonts,amsgen,amsmath,amstext,amsbsy,amsopn,amsthm
}
\usepackage{ebezier,eepic}
\usepackage{color}
\usepackage{multirow}
\usepackage{mathrsfs}
\usepackage{tikz}
\usepackage{enumerate}
\newtheorem{dfn}{Definition}
\newtheorem{thm}[dfn]{Theorem}

\newtheorem{prop}[dfn]{Proposition}

\newtheorem{corollary}[dfn]{Corollary}

\begin{document}

\title{Decomposing $C_4$-free graphs under degree constraints}

\date{}

\author{
Jie Ma\thanks{School of Mathematical Sciences, University of Science and Technology of China,
Hefei, Anhui 230026, China. Email: jiema@ustc.edu.cn. Partially supported by NSFC grants 11501539 and 11622110.}
~~~~~
Tianchi Yang\thanks{School of Mathematical Sciences, University of Science and Technology of China,
Hefei, Anhui 230026, China. Email: ytc@mail.ustc.edu.cn.}
}

\maketitle

\begin{abstract}
A celebrated theorem of Stiebitz \cite{Sti} asserts that any graph with minimum degree at least $s+t+1$ can be partitioned into two parts
which induce two subgraphs with minimum degree at least $s$ and $t$, respectively.
This resolved a conjecture of Thomassen.
In this paper, we prove that for $s,t\geq 2$, if a graph $G$ contains no cycle of length four and has minimum degree at least $s+t-1$, then
$G$ can be partitioned into two parts which induce two subgraphs with minimum degree at least $s$ and $t$, respectively.
This improves the result of Diwan in \cite{Diwan}, where he proved the same statement for graphs of girth at least five.
Our proof also works for the case of variable functions, in which the bounds are sharp as showing by some polarity graphs.
As a corollary, it follows that any graph containing no cycle of length four with minimum degree at least $k+1$ contains $k$ vertex-disjoint cycles.
\end{abstract}


\section{Introduction}

All graphs $G=(V,E)$ considered here are finite and simple.
The degree of a vertex $v$ in $G$ is expressed as $d_G(v)$, and for a subset $A\subseteq V$,
we denote by $d_A(v)$ the number of vertices in $A$ that are adjacent to $v$ in $G$.
By a {\it partition} $(A,B)$ of $V$, we mean that $A,B$ are two disjoint non-empty sets with $A\cup B=V$.

Many problems raised in graph theory concern graph decompositions under certain constraints (for instance, graph coloring problems).
Perhaps one of the earliest results regrading graph decompositions under degree constraints is due to Lov\'asz \cite{Lov} in 1966,
who proved that any graph with maximum degree at most $s+t+1$ has a partition $(A,B)$ such that
the subgraphs induced on $A$ and $B$ have maximum degree at most $s$ and $t$, respectively.
This was generalized by Borodin and Kostochka \cite{BK77} to the case of variable functions
(the meaning of which will be clear from the contents later).

The counterpart of Lov\'asz' theorem, i.e., graph decompositions under minimum degree constraints, also has received extensive research.
Let $f(s,t)$ be the least function such that any graph with minimum degree at least $f(s,t)$ has a partition $(A,B)$ so that
the subgraphs induced on $A$ and $B$ have minimum degree at least $s$ and $t$, respectively.
The existence of $f(s,t)$ was proved by Thomassen \cite{T83} in 1983,
and then this function was subsequently improved by H\"aggkvist, Alon, and Hajnal \cite{Haj} (see the discussion in \cite{T88}).
It was also conjectured by Thomassen \cite{T83,T88} that $f(s,t)=s+t+1$, and complete graphs show that this bound would be tight.
Later, Stiebitz \cite{Sti} resolved this conjecture completely.
In fact he proved the following stronger result, in the setting of variable functions.
Let $\bf{N}$ denote the set of non-negative integers.

\begin{thm}\label{Thm:Sti} (Stiebitz \cite{Sti})
Let $G$ be a graph and $a,b: V(G)\rightarrow {\bf{N}}$ be two functions.
If $d_G(x)\geq a(x)+b(x)+1$ for every vertex $x\in V(G)$,
then there is a partition $(A,B)$ of $V(G)$ satisfying that
$d_A(x)\geq a(x)$ for every $x\in A$, and
$d_B(x)\geq b(x)$ for every $x\in B$.
\end{thm}

Kaneko \cite{Kan} proved that any triangle-free graph with minimum degree at least $s+t$ can already force a partition $(A,B)$ as above.
The minimum degree condition was further sharpen by Diwan \cite{Diwan}, when cycles of length four are also forbidden.
To be precise, Diwan proved that, assuming $s,t\geq 2$, any graph of girth at least five with minimum degree at least $s+t-1$ has a partition $(A,B)$ such that
the subgraphs induced on $A$ and $B$ have minimum degree at least $s$ and $t$, respectively.
For related problems on graph decompositions with degree constraints or other variances,
we refer readers to \cite{Ban,BL16,BTV,GK04,LFW,SS17,Sti17}.

In this paper the following result is proved.

\begin{thm}\label{Thm: main}
Let $G$ be a graph containing no cycles of length four and $a,b: V(G)\rightarrow {\bf{N}}_{\geq 2}$ be two functions,
where ${\bf{N}}_{\geq 2}$ denotes the set of integers at least two.
If $$d_G(x)\geq a(x)+b(x)-1$$ for every vertex $x\in V(G)$,
then there is a partition $(A,B)$ of $V(G)$ satisfying that
$d_A(x)\geq a(x)$ for every $x\in A$, and $d_B(x)\geq b(x)$ for every $x\in B$.
\end{thm}

This is tight in the following two perspectives.
First, the ranges of the functions $a, b$ cannot be relaxed to the set of integers at least one by the following example:
Take any $d$-regular connected graph $G$ and the constant functions $a=1$ and $b=d$;
then it is easy to see that none of the partitions $(A,B)$ could satisfy the properties.
Second, one also cannot lower the degree condition further by the following proposition.

\begin{prop}\label{Prop}
There exist a graph $G$, which contains no cycle of length four, and two functions $a, b: V(G)\rightarrow {\bf{N}}_{\geq 2}$
such that $d_G(x)=a(x)+b(x)-2$ for every vertex $x\in V(G)$ and moreover, for any partition $(A,B)$ of $V(G)$,
there is either a vertex $x\in A$ with $d_A(x)<a(x)$ or a vertex $x\in B$ with $d_B(x)<b(x)$.
\end{prop}

When choosing $a,b$ as constants functions in Theorem \ref{Thm: main}, it strengthens Diwan's result to graphs
containing no cycles of length four, instead of graphs with girth at least five.
We state this in a more general version, namely for $k$-partitions where $k\geq 2$.

\begin{corollary}
Let $s_1,..., s_k\geq 2$ be integers. Any graph containing no cycles of length four with minimum degree at least $s_1+...+s_k-(k-1)$
can be partitioned into $k$ parts such that the subgraphs induced on the $k$ parts have minimum degree at least $s_1,..., s_k$, respectively.
\end{corollary}

This also can be used for finding vertex-disjoint cycles in graphs with high minimum degree.
It is known (see \cite{T83-girth}) that if a graph has girth at least five and minimum degree at least $k+1$,
then it contains $k$ vertex-disjoint cycles.
By choosing $s_1, ..., s_k$ to be two in the above statement, we can obtain the following.

\begin{corollary}
Any graph containing no cycles of length four with minimum degree at least $k+1$ contains $k$ vertex-disjoint cycles.
\end{corollary}

The rest of the paper is organized as follows. In Section 2,
we introduce some notations and show several propositions (including Proposition \ref{Prop}).
In Section 3 we complete the proof of Theorem \ref{Thm: main}.

\section{Notations and propositions}
Let $G$ be a graph and $f:V(G)\rightarrow \bf{N}$ be a function.
We say that $G$ is {\it $f$-degenerate} if for every subgraph $H$ of $G$ there is a vertex $u$ such that $d_{H}(u)\leq f(v)$.
For a subset $A\subseteq V(G)$, we say that $A$ is {\it $f$-degenerate} if $G[A]$ is $f$-degenerate,
and it is {\it $f$-good} if for every vertex $u\in A$, $d_A(u)\geq f(u)$.
A vertex $u$ is called an {\it $f$-vertex in $A$} if $u\in A$ and $d_A(u)=f(u)$.
It is immediate from the definitions that
\begin{prop}\label{Prop:1}
A subset $A$ of $V(G)$ does not contain any $f$-good subset if and only if it is $(f-1)$-degenerate.
\end{prop}
\noindent We point out that this fact will be repeatedly used in the coming proofs.

Let $a, b:  V(G)\rightarrow \bf{N}$ be two functions.
We call a pair $(A,B)$ of disjoint subsets of $V(G)$ as a {\it feasible pair} (with respect to $a,b$) if $A$ is $a$-good and $B$ is $b$-good.
If in addition $(A,B)$ is a partition of $V(G)$, then we also call it a {\it feasible partition}.
The following nice property was first proved in \cite{Sti}.
We give a proof for the completeness.

\begin{prop}(\cite{Sti})\label{Prop:2}
Assume that for any $x\in V(G)$, $d_G(x)\ge a(x)+b(x)-1$.
If $G$ has a feasible pair, then $G$ also has a feasible partition.
\end{prop}
\begin{proof}
Choose a feasible pair $(A,B)$ in $G$ such that $A\cup B$ is maximal.
We show that $(A,B)$ must be a partition.
Suppose $C=V(G)\backslash (A\cup B)$ is non-empty.
Then any $x\in C$ satisfies that $d_A(x)\leq a(x)-1$, as otherwise $(A\cup \{x\}, B)$ is also feasible.
This implies that for any $x\in C$, $d_{B\cup C}(x)\geq b(x)$ and thus $B\cup C$ is $b$-good,
completing the proof.
\end{proof}

We now prove Proposition \ref{Prop}.

\medskip

{\noindent \it Proof of Proposition \ref{Prop}.}
One such example is a triangle with constant functions $a=b=2$.
We provide other examples by considering Erd\H{o}s-Renyi polarity graphs.
Let $q$ be a prime and $V$ be a 3-dimensional vector space over $\mathbb{F}_q$.
The Erd\H{o}s-Renyi polarity graph $ER_q$ is a simple graph whose vertices are the 1-dimensional subspaces $[\vec{v}]$ in $V$,
where two vertices $[\vec{v}]$ and $[\vec{w}]$ are adjacent if and only if the vectors $\vec{v}$ and $\vec{w}$ are orthogonal.
It is well-known that $ER_q$ contains no cycles of length four.

Let $q=3$. Then the graph $ER_3$ has $q^2+q+1=13$ vertices as well as the following properties.
If we let $S$ be the set of vertices of degree 3 in $ER_3$,
then $S$ is an independent set of size 4 and $T:=V(ER_3)\backslash S$ consists of all vertices of degree 4.
Moreover, the induced subgraph on $T$ is connected and its edges can be partitioned into four edge-disjoint triangles.
Choose functions $a,b$ such that $a$ is the constant function three and $b(x)=2$ if $x\in S$ and $b(x)=3$ if $x\in T$.
Then for every vertex $x$ in $ER_3$ it holds that $d_{ER_3}(x)=a(x)+b(x)-2$.
It remains to show that there is no feasible partition in $ER_3$.
Suppose for a contradiction that there exists a feasible partition $(A,B)$.
We claim that the vertices of any triangle $uvw$ in $T$ must belong to the same part.
Indeed, if $u$ is in one part and $v, w$ are in another part,
then $u\in T$ has at most two neighbors in its own part,
contradicting that $(A,B)$ is feasible with respect to the functions $a, b$.
Since the induced subgraph on $T$ is connected,
this implies that $T$ is contained in one part, say $A$.
Then $B$ is just a subset of $S$, which is an independent set, a contradiction.
This completes the proof that $ER_3$ and the so-defined functions $a,b$ serve as an example to this proposition.

Using similar arguments, one can show that $ER_2$ (plus some proper functions $a,b$) is also an example for this proposition.
\qed

\medskip

In the coming proof, we sometime adopt the notations $u\sim v$ and $u\not\sim v$ to express the situation
that the vertices $u,v$ are adjacent or not, respectively.

\section{The proof of Theorem \ref{Thm: main}}
Throughout this section, let $G$ be a graph which contains no cycles of length four
and $a, b:  V(G)\rightarrow {\bf{N}}_{\geq 2}$ be two functions
such that for any $x\in V(G)$, $d_G(x)\geq a(x)+b(x)-1$.
Suppose for a contradiction that $G$ contains no feasible partitions.
By Proposition \ref{Prop:2}, we have the fact that
\begin{equation}\label{equ:no pair}
\text{there is no feasible pairs in } G.
\end{equation}

Our proof proceeds with a sequence of claims.

\medskip

\noindent {\bf Claim 1.} It suffices to assume that for any $x\in V(G)$, $d_G(x)=a(x)+b(x)-1$.

\begin{proof}
Indeed we may increase $a,b$ to get functions $a',b'$ such that $a'\geq a, ~b'\geq b$
and $d_G(x)=a'(x)+b'(x)-1$ for all $x\in V(G)$.
Now suppose Theorem \ref{Thm: main} holds under the assumption of these inequalities.
As $a'\geq a$ and $b'\geq b$, any feasible partition of $V(G)$ with respect to $a', b'$
is also a feasible partition with respect to $a, b$. This proves Claim 1.
\end{proof}

\noindent {\bf Definition 1.} A partition $(A,B)$ of $V(G)$ is an {\it $(a,b)$-partition} if $A$ is $(a-1)$-degenerate and $B$ is $(b-1)$-degenerate.

\medskip

\noindent {\bf Claim 2.} There exist $(a,b)$-partitions in $G$.

\begin{proof}
Consider a minimal $a$-good subset $A\subseteq V(G)$ (note that such subsets exist, as $V(G)$ is one).
So $|A|\geq 2$. Let $B=V\backslash A$.
Clearly $B$ contains no $b$-good subsets, as otherwise there exist feasible pairs, contradicting \eqref{equ:no pair}.
So, by Proposition \ref{Prop:1}, $B$ is $(b-1)$-degenerate.
By the minimality of $A$, there exists a vertex $x\in A$ with $d_A(x)=a(x)$.
By Claim 1, $d_B(x)=b(x)-1$ and thus $B\cup \{x\}$ is also $(b-1)$-degenerate.
On the other hand, $A\backslash\{x\}$ is non-empty and $(a-1)$-degenerate (again by the minimality of $A$).
So $(A\backslash\{x\}, B\cup \{x\})$ gives an $(a,b)$-partition.
\end{proof}

For any partition $(A,B)$ in $G$, we define a weight function as following:
\begin{equation}\label{equ:w(A,B)}
w(A,B):=E(G[A])+E(G[B])+\sum_{x\in A}b(x)+\sum_{x\in B}a(x).
\end{equation}

\noindent {\bf Claim 3.} For a partition $(A,B)$, let $u\in A, v\in B$ be two vertices
such that $d_A(u)=a(u)-\alpha$ and $d_B(v)=b(v)-\beta$.
Let $\delta=1$ if $u,v$ are adjacent and $\delta=0$ otherwise.
Then $$w(A\backslash \{u\}, B\cup \{u\})-w(A,B)=2\alpha -1,$$
$$w(A\cup \{v\}\backslash \{u\}, B\cup\{u\}\backslash \{v\})-w(A,B)=2(\alpha+\beta-1-\delta).$$

\begin{proof}
This follows directly from the definition and Claim 1.
We only show the second identity. Its left hand side equals
\begin{equation*}
-d_A(u)-d_B(v)-b(u)-a(v)+(d_B(u)-\delta)+(d_A(v)-\delta)+b(v)+a(u).
\end{equation*}
After simplifying, this gives $2(\alpha+\beta-1-\delta)$.
\end{proof}

\noindent {\bf Definition 2.} Let $\mathscr{P}$ be the family consisting of all $(a,b)$-partitions $(A,B)$,
which attain the maximum weight $w(A,B)$ among all $(a,b)$-partitions in $G$.
For any $(A,B)\in \mathscr{P}$, define $$A^*=\{x\in A~|~d_A(x)\leq a(x)-1\} \text{ and } B^*=\{x\in B~|~d_B(x) \leq b(x)-1\}.$$

It is easy to see that both $A^*$ and $B^*$ are non-empty. So for any $x\in B^*$,
we have $|A|\geq d_A(x)\geq a(x)\geq 2$. Hence, we see that both $A$ and $B$ contain at least two vertices.

\medskip

\noindent \textbf{Claim 4.} For any $(A,B)\in \mathscr{P}$, every vertex in $A^*$ is adjacent to every vertex in $B^*$.

\begin{proof}
Suppose that there exist non-adjacent vertices $u\in A^*$ and $v\in B^*$.
Let $d_A(u)=a(u)-\alpha$ and $d_B(v)=b(v)-\beta$. So $\alpha, \beta\geq 1$.

First consider $u\in A$. We have $|A|\geq 2$, so $A\backslash \{u\}$ is non-empty.
By Claim 3, $w(A\backslash \{u\},B\cup \{u\})-w(A,B)=2\alpha-1\geq 1$,
thus $(A\backslash \{u\},B\cup \{u\})$ cannot be an $(a,b)$-partition.
Since $A\backslash \{u\}$ is $(a-1)$-degenerate,
this implies that $B\cup \{u\}$ cannot be $(b-1)$-degenerate.
Therefore there exists a $b$-good subset $B'\subseteq B\cup \{u\}$.
As $u,v$ are not adjacent, $d_{B\cup \{u\}}(v)=d_B(v)\leq b(v)-1$, so $B'\subseteq B\cup \{u\}\backslash \{v\}.$

By considering $v\in B$, similarly we can find an $a$-good subset $A'\subseteq A\cup \{v\}\backslash \{u\}$.
Then, $(A',B')$ forms a feasible pair, a contradiction to \eqref{equ:no pair}.
The proof of Claim 4 is completed.
\end{proof}

\noindent \textbf{Claim 5}. For any $(A,B)\in \mathscr{P}$, either $A^*$ or $B^*$ consists of exactly one vertex.
Moreover, every vertex in $V(G)\backslash A^*$ is adjacent to at most one vertex in $B^*$ and
every vertex in $V(G)\backslash B^*$ is adjacent to at most one vertex in $A^*$.

\begin{proof}
Otherwise there is some $C_4$ by Claim 4 (note that both $A^*, B^*$ are non-empty).
\end{proof}

\noindent \textbf{Claim 6}. For any $(A,B)\in \mathscr{P}$, $u\in A^*$ and $v\in B^*$,
we have $d_A(u)=a(u)-1,d_B(v)=b(v)-1$ and $(A\cup \{v\}\backslash \{u\},B\cup \{u\}\backslash \{v\})\in \mathscr{P}$.
\begin{proof}
Let $d_A(u)=a(u)-\alpha$ and $d_B(v)=b(v)-\beta$, where $\alpha, \beta\geq 1$.

We first show that $(A\cup \{v\}\backslash \{u\},B\cup \{u\}\backslash \{v\})$ is an $(a,b)$-partition.
Suppose not. Without loss of generality, we may assume that there exists a $b$-good subset $B'\subseteq B\cup \{u\}\backslash \{v\}$.
Then we must have $u\in B'$. If $A\cup \{v\}$ is $(a-1)$-degenerate, then $(A\cup \{v\}, B\backslash \{v\})$ is an $(a,b)$-partition
and by Claim 3, $w(A\cup \{v\}, B\backslash \{v\})-w(A,B)=2\beta-1\geq 1$, a contradiction.
Therefore there exists an $a$-good subset $A'\subseteq A\cup \{v\}$.
If $u\notin A'$, then $(A',B')$ is a feasible pair, a contradiction.
So $u\in A'$. Then the only possibility is that $d_{A}(u)= a(u)-1$.
This also shows $d_B(u)=b(u)$ and thus $d_{B\cup \{u\}\backslash \{v\}}(u)=b(u)-1$, contradicting with $u\in B'$.
So indeed $(A\cup \{v\}\backslash \{u\},B\cup \{u\}\backslash \{v\})$ is an $(a,b)$-partition.

By Claim 3 again, $w(A\cup \{v\}\backslash \{u\},B\cup \{u\}\backslash \{v\})-w(A,B)=2(\alpha+\beta-2)\geq 0$.
By the maximality of $w(A,B)$, $\alpha=\beta=1$. So $(A\cup \{v\}\backslash \{u\},B\cup \{u\}\backslash \{v\})\in \mathscr{P}$.
\end{proof}

\noindent \textbf{Claim 7}. For any $(A,B)\in \mathscr{P}$, $|A\backslash A^*|\geq 2$ and $|B\backslash B^*|\geq 2$.
\begin{proof}
By Claims 5 and 6, we may assume $B^*=\{v\}$ and $d_B(v)= b(v)-1\geq 1$.
Choose any $v_1\in B\backslash B^*$. Since $d_B(v_1)\geq b(v_1)\geq 2$,
there exists a neighbor of $v_1$ in $B\backslash B^*$. So $|B\backslash B^*|\geq 2$.
Similarly, if $|A^*|=1$, then we also have $|A\backslash A^*|\geq 2$.
Assume $|A^*|\geq 2$.

If there exists some vertex $u_1\in A\backslash A^*$,
then $d_A(u_1)\geq a(u_1)\geq 2$. By Claim 5, $u_1$ has at most one neighbor in $A^*$ and
thus at least one neighbor in $A\backslash A^*$, therefore $|A\backslash A^*|\geq 2$.

So we may assume $A=A^*=\{u_1,...,u_\ell\}$.
By Claim 6, $d_A(u_i)=a(u_i)-1\geq 1$. This, together with Claim 5,
shows that in fact any $u_i$ has exact one neighbour in $A$ and $a(u_i)=2$.
Since all vertices in $A$ are adjacent to $v$,
we see that $A\cup \{v\}$ induces a union of triangles which pairwise intersect at $v$.
As $d_A(v)=a(v)$, $A\cup\{v\}$ is $a$-good.
For any $x\in B\backslash \{v\}$, there is at most one neighbor of $x$ in $A\cup \{v\}$, as otherwise there is a $C_4$.
So $d_{B\backslash \{v\}}(x)\geq a(x)+b(x)-2\geq b(x)$.
We then find a feasible partition $(A\cup \{v\}, B\backslash \{v\})$.
\end{proof}

\noindent {\bf Definition 3.} For any $(A,B)\in \mathscr{P}$,
we define $$A^\diamond=\{u\in A\backslash A^*~| ~d_{A\backslash A^*}(u)\leq a(u)-1\}
\text{~ and ~} B^\diamond=\{v\in B\backslash B^*~| ~d_{B\backslash B^*}(v)\leq a(v)-1\}.$$

\noindent \textbf{Claim 8}. For any $(A,B)\in \mathscr{P}$, the subsets $A^\diamond$ and $B^\diamond$ are non-empty.
And any $u\in A^\diamond$ has exactly one neighbor in $A^*$ and $d_A(u)=a(u)$;
similarly, any $v\in B^\diamond$ has exactly one neighbor in $B^*$ and $d_B(v)=b(v)$.

\begin{proof}
Claim 7 shows that $A\backslash A^*$ and $B\backslash B^*$ induce two non-empty subgraphs,
which are $(a-1)$-degenerate and $(b-1)$-degenerate, respectively.
So $A^\diamond$ and $B^\diamond$ are non-empty.
It suffices to consider $u\in A^\diamond$.
By Claim 5, $u$ has at most one neighbor in $A^*$ and thus $d_A(u)\leq a(u)$.
But $u\notin A^*$, which means $d_A(u)\geq a(u)$.
This shows that $d_{A}(u)=a(u)$ and $u$ has exactly one neighbor in $A^*$.
\end{proof}

\noindent \textbf{Claim 9}. For any $(A,B)\in \mathscr{P}$, there exists one of the following five configurations in $A$ (see Figure 1):
\begin{itemize}
\item [(A1)] two $a$-vertices $u_1,u_2$ in $A$ are adjacent to the same vertex $u\in A^*$,

\item [(A2)] two $a$-vertices $u_1,u_2$ in $A$ are adjacent to $u, u'\in A^*$, respectively,

\item [(A3)] there exist two $a$-vertices $u_1,u_2$ in $A$ and a vertex $u\in A^*$ such that $u_1\sim u_2$, $u_1\sim u$ and $u_2\not\sim u$,

\item [(A4)] there exist an $a$-vertex $u_1$ in $A$, an $(a+1)$-vertex $u_2$ in $A$ and a vertex $u\in A^*$ such that $u_1, u_2, u$ form a triangle, and
\item [(A5)] there exist an $a$-vertex $u_1$ in $A$, an $(a+1)$-vertex $u_2$ in $A$ and two vertices $u, u'\in A^*$ such that $u_1\sim u_2$, $u_1\sim u$ and $u_2\sim u'$.
\end{itemize}

\begin{figure}[ht!]\label{fig:A}
	\begin{minipage}{0.5\linewidth}
		\centering
		\begin{tikzpicture}
		[scale=0.7, auto=left,   roundnode/.style={circle, draw, fill=black ,
			inner sep=0pt, minimum width=4pt}]
		\draw[gray, thick] (-3,0) rectangle (3,2.5) ;
		\draw[gray, thick] (-3,2.5) rectangle (3,5);
		\filldraw[black] (-3,4.5) circle (0pt) node[anchor=west]  {\large $A^*$ };
		\filldraw[black] (-3,2 ) circle (0pt) node[anchor=west]  {\large $A\backslash A^*$ };
		\node[roundnode](u) at (0,3.7)[label=north:$u$]{};
		\node[roundnode] (u1) at (-1,1.2) [label=south:$u_1$]{};
		\node[roundnode] (u2) at (1, 1.2) [label=south:$u_2$]{};
		\foreach \from/\to in {u/u1, u/u2}
		\draw[black,very thick] (\from) -- (\to);
		\end{tikzpicture}
		\caption*{(A1)}
\medskip
	\end{minipage}
	\begin{minipage}{0.5\linewidth}
		\centering
		\begin{tikzpicture}
		[scale=0.7, auto=left,   roundnode/.style={circle, draw, fill=black ,
			inner sep=0pt, minimum width=4pt}]
		\draw[gray, thick] (-3,0) rectangle (3,2.5) ;
		\draw[gray, thick] (-3,2.5) rectangle (3,5);
		\filldraw[black] (-3,4.5) circle (0pt) node[anchor=west]  {\large $A^*$ };
		\filldraw[black] (-3,2 ) circle (0pt) node[anchor=west]  {\large $A\backslash A^*$ };
		\node[roundnode](u) at (-0.5,3.7)[label=north:$u$]{};
		\node[roundnode](u') at (1.5,3.7)[label=north:$u'$]{};
		\node[roundnode] (u1) at (-1,1.2) [label=south:$u_1$]{};
		\node[roundnode] (u2) at (1, 1.2) [label=south:$u_2$]{};
		\foreach \from/\to in {u/u1, u'/u2}
		\draw[black,very thick] (\from) -- (\to);
		\end{tikzpicture}
		\caption*{(A2)}
\medskip
	\end{minipage}
	
	
	\begin{minipage}{0.32\linewidth}
		\centering
		\begin{tikzpicture}
		[scale=0.7, auto=left,   roundnode/.style={circle, draw, fill=black ,
			inner sep=0pt, minimum width=4pt}]
		\draw[gray, thick] (-3,0) rectangle (3,2.5) ;
		\draw[gray, thick] (-3,2.5) rectangle (3,5);
		\filldraw[black] (-3,4.5) circle (0pt) node[anchor=west]  {\large $A^*$ };
		\filldraw[black] (-3,2 ) circle (0pt) node[anchor=west]  {\large $A\backslash A^*$ };
		\node[roundnode](u) at (0,3.7)[label=north:$u$]{};
		\node[roundnode] (u2) at (-1,1.2) [label=south:$u_2$]{};
		\node[roundnode] (u1) at (1, 1.2) [label=south:$u_1$]{};
		\foreach \from/\to in {u/u1, u2/u1}
		\draw[black,very thick] (\from) -- (\to);
		\end{tikzpicture}
		\caption*{(A3)}
	\end{minipage}
	\begin{minipage}{0.32\linewidth}
		\centering
		\begin{tikzpicture}
		[scale=0.7, auto=left,   roundnode/.style={circle, draw, fill=black ,
			inner sep=0pt, minimum width=4pt}]
		\draw[gray, thick] (-3,0) rectangle (3,2.5) ;
		\draw[gray, thick] (-3,2.5) rectangle (3,5);
		\filldraw[black] (-3,4.5) circle (0pt) node[anchor=west]  {\large $A^*$ };
		\filldraw[black] (-3,2 ) circle (0pt) node[anchor=west]  {\large $A\backslash A^*$ };
		\node[roundnode](u) at (0,3.7)[label=north:$u$]{};
		\node[roundnode] (u2) at (-1,1.2) [label=south:$u_2$]{};
		\node[roundnode] (u1) at (1, 1.2) [label=south:$u_1$]{};
		\foreach \from/\to in {u/u2, u/u1,u2/u1}
		\draw[black,very thick] (\from) -- (\to);
		\end{tikzpicture}
		\caption*{(A4)}
	\end{minipage}
	\begin{minipage}{0.32\linewidth}
		\centering
		\begin{tikzpicture}
		[scale=0.7, auto=left,   roundnode/.style={circle, draw, fill=black ,
			inner sep=0pt, minimum width=4pt}]
		\draw[gray, thick] (-3,0) rectangle (3,2.5) ;
		\draw[gray, thick] (-3,2.5) rectangle (3,5);
		\filldraw[black] (-3,4.5) circle (0pt) node[anchor=west]  {\large $A^*$ };
		\filldraw[black] (-3,2 ) circle (0pt) node[anchor=west]  {\large $A\backslash A^*$ };
		\node[roundnode](u') at (-0.5,3.7)[label=north:$u'$]{};
		\node[roundnode](u) at (1.5,3.7)[label=north:$u$]{};
		\node[roundnode] (u2) at (-1,1.2) [label=south:$u_2$]{};
		\node[roundnode] (u1) at (1, 1.2) [label=south:$u_1$]{};
		\foreach \from/\to in {u'/u2, u/u1,u2/u1}
		\draw[black,very thick] (\from) -- (\to);
		\end{tikzpicture}
		\caption*{(A5)}
	\end{minipage}
	\caption{The five configurations in $A$}
\end{figure}

And the analog also holds for $B$ (call the five configurations as (B1)-(B5), respectively).

\begin{proof}
If $A^\diamond$ has at least two vertices (say $u_1, u_2$), then by Claim 8, each of $u_1, u_2$ has exactly one neighbor in $A^*$.
This leads to the configuration (A1) or (A2).

If $A^\diamond$ has exactly one vertex (say $u_1$), then by Claim 7,
$A\backslash (A^*\cup \{u_1\})$ is non-empty and also $(a-1)$-degenerate.
Then $u_1$ has a neighbour $u_2\in A\backslash (A^*\cup \{u_1\})$ satisfying that $d_{A\backslash A^*}(u_2)= a(u_2)$.
This leads to three possible configurations: (A3) when $u_2$ has no neighbour in $A^*$, (A4) when $u_1,u_2$ have the same neighbour in $A^*$, and (A5) when $u_1,u_2$ have different neighbours in $A^*$. This proves Claim 9.
\end{proof}

\noindent {\bf Definition 4.} For any $(A,B)\in \mathscr{P}$, a path $u_1\sim u\sim v\sim v_1$ is called a {\it special path},
if $u\in A^*$, $v\in B^*$, $u_1$ is an $a$-vertex in $A$, and $v_1$ is a $b$-vertex in $B$.

\medskip

\noindent \textbf{Claim 10}. For any special path $u_1\sim u\sim v\sim v_1$, either $u_1v\in E(G)$ or $v_1u\in E(G)$.

\begin{proof}
Suppose that $u_1v, v_1u\notin E(G)$.
Let $(A',B')$ be the new partition obtained from $(A,B)$ by exchanging $u$ and $v$.
By Claim 6, $(A',B')\in \mathscr{P}$.
Also $u_1$ becomes an $(a-1)$-vertex in $A'$ and $v_1$ becomes a $(b-1)$-vertex in $B'$.
Then by Claim 4, we have $u_1v_1\in E(G)$. So $u_1,u,v,v_1$ form a cycle of length four, a contradiction.
\end{proof}

Now let us fix a partition $(A,B)\in \mathscr{P}$.
So by Claim 9, there exist two configurations, say (Ai) in $A$ and (Bj) in $B$.
In what follows, we will finish the proof by showing that
any combination of (Ai) and (Bj) for all $1\leq i, j\leq 5$ will derive some contradiction
(either finding a cycle of length four or contradicting the above claims).

Take the vertex $u\in A^*$ and the $a$-vertex $u_1$ in $A$ from Claim 9;
and call the analogous vertices of $u, u_1$ in $B$ as $v, v_1$, respectively.
Note that in any situation, we have that $u\in A^*, v\in B^*$, $u_1$ is an $a$-vertex in $A$ and $v_1$ is a $b$-vertex in $B$.
Therefore, $u_1\sim u\sim v\sim v_1$ is a special path for $(A,B)\in \mathscr{P}$.
By Claim 10, we have either $u_1v\in E(G)$ or $v_1u\in E(G)$.
Without loss of generality,
\begin{equation}\label{equ:u1v}
\text{we assume that } u_1v\in E(G) \text{ and } v_1u\notin E(G).
\end{equation}
If the configuration (A4) or (A5) occurs, then \eqref{equ:u1v} will force a $C_4$, a contradiction.
Therefore, there are only 3 configurations left (under the assumption \eqref{equ:u1v}), namely (A1), (A2) or (A3).
We distinguish among these three cases.

\bigskip

\noindent \textbf{Case 1:} Configuration (A1) occurs.

\medskip

We see that $u_2\sim u\sim v\sim v_1$ is a special path. By Claim 10, either $u_2v\in E(G)$ or $uv_1\in E(G)$.
If $u_2v\in E(G)$, then $u_1,u_2,u,v$ form a $C_4$ and if $uv_1\in E(G)$, then $u_1,v_1,u,v$ form a $C_4$.
This shows that under the assumption \eqref{equ:u1v}, (A1) does not occurs.

\medskip

\noindent \textbf{Case 2:} Configuration (A2) occurs.

\medskip

In this case, we will show that either there exists a $C_4$ or this can be reduced to the configuration (A3).
Note that we have $|A^*|\geq 2$. So $B^*=\{v\}$. So only the configurations (B1), (B3), and (B4) can occur in $B$.

First suppose that (B3) occurs.\footnote{In this paragraph we will only use the vertices $u_1,u\in A$,
so this shows that under the assumption \eqref{equ:u1v}, (B3) does not occur no matter which configuration is in $A$.}
Then there is another $b$-vertex $v_2$ (other than $v_1$) in $B$ such that $v_2\sim v_1$ and $v_2\not\sim v$.
Let $(A',B')$ be the partition obtained from $(A,B)$ by exchanging $u$ and $v_1$.
We may easily infer that $uv_1, uv_2, u_1v_1, u_1v_2\notin E(G)$ (as otherwise there is a $C_4$).
So $u_1,v_1$ are $(a-1)$-vertices in $A'$, $v,v_2$ are $(b-1)$-vertices in $B'$, and $u$ is a $b$-vertex in $B'$.
We claim that $(A',B')\in \mathscr{P}$.
We first observe that $A'$ is $(a-1)$-degenerate;
otherwise, as $A$ contains no $a$-good subsets,
there must exist an $a$-good subset in $A'$ which contains $v_1$,
but $v_1$ is an $(a-1)$-vertex in $A'$, a contradiction.
If there exists a $b$-good subset $B''\subseteq B'$, then similarly $u\in B''$ and so $d_{B''}(u)=d_{B'}(u)=b(u)$,
which shows that all neighbors of $u$ in $B'$ should also belong to $B''$.
But the neighbor $v$ of $u$ is a $(b-1)$-vertex in $B'$, a contradiction.
So $B'$ is $(b-1)$-degenerate and thus $(A',B')$ is an $(a,b)$-partition.
By Claim 3, we also have $w(A',B')-w(A,B)=0$. This proves that $(A',B')\in \mathscr{P}$.
Then by Claim 4, $u_1,v_1,v,v_2$ give a $C_4$. This shows that (B3) does not occur.

Now we consider when (B1) or (B4) occurs.
We claim that all vertices in $A^\diamond$ are adjacent to $v$.
Consider any vertex $w\in A^\diamond\backslash \{u_1\}$ and assume $wv\notin E(G)$.
By Claim 8, $w$ is an $a$-vertex in $A$ and adjacent to exactly one vertex in $A^*$ (say $w'$).
If $w'=u$, then the configuration (A1) occurs.
So we have $w'\neq u$. Then the special path $w\sim w'\sim v\sim v_1$ forces either $wv\in E(G)$ or $w'v_1\in E(G)$.
So $w'v_1\in E(G)$. In (B4), $w',v,v_1,v_2$ will form a $C_4$.
Now let us consider (B1), where $v_2$ is a $b$-vertex in $B$ and $v_2v\in E(G)$.
Then $w\sim w'\sim v\sim v_2$ is also a special path.
As $wv\notin E(G)$, we must have $w'v_2\in E(G)$,
which again gives a $C_4$ (with vertices $w',v,v_1,v_2$). This proves the claim.

We see that all vertices in $A^*\cup A^\diamond$ are adjacent to $v$ and
thus any vertex in $A$ has at most one neighbor in $A^*\cup A^\diamond$ (otherwise, there is a $C_4$).
This implies that $A\backslash (A^*\cup A^\diamond)\neq \emptyset$,
as otherwise any vertex $w\in A^\diamond$ has $d_{A^*\cup A^\diamond}(w)=d_A(w)=a(w)\geq 2$, a contradiction.
Thus, there exists a vertex $x\in A\backslash (A^*\cup A^\diamond)$ with $d_{A\backslash (A^*\cup A^\diamond)}(x)\leq a(x)-1$.
But $d_{A}(x)\geq a(x)$ (as $x\notin A^*$) and $x$ has at most one neighbor in $A^*\cup A^\diamond$.
This shows that $x$ is an $a$-vertex in $A$ and has exactly one neighbor (say $x'$) in $A^*\cup A^\diamond$.
Also as $x\notin A^\diamond$, we have $d_{A\backslash A^*}(x)\geq a(x)$, which shows that $x'\in A^\diamond$.
Let $x''$ be the unique neighbor of $x'$ in $A^*$ (by Claim 8).
Now the three vertices $x,x',x''$ give the configuration (A3) in $A$.
Note that we also have $x'v\in E(G)$ and $v_1x\notin E(G)$ (i.e., the equivalent assumption as \eqref{equ:u1v}).
Therefore, it suffices to consider the following case.

\medskip

\noindent \textbf{Case 3:} Configuration (A3) occurs.

\medskip

There are 5 configurations in $B$ to consider.

\begin{figure}[ht!]
	\begin{minipage}{0.32\linewidth}
		\centering
		\begin{tikzpicture}
		[scale=0.6, auto=left,   roundnode/.style={circle, draw, fill=black ,inner sep=0pt, minimum width=4pt}]
		\draw[gray, thick] (-3,0) rectangle (5,2.5) ;
		\draw[gray, thick] (-3,2.5) rectangle (5,5);
		\draw[gray, thick] (1,0)--(1,5);
		\filldraw[black] (-3,4.5) circle (0pt) node[anchor=west]  {\small $A^*$ };
		\filldraw[black] (-3,2 ) circle (0pt) node[anchor=west]  {\small $A\backslash A^*$ };
		\filldraw[black] (3.9,4.5) circle (0pt) node[anchor=west]  {\small $B^*$ };
		\filldraw[black] (3.3,2 ) circle (0pt) node[anchor=west]  {\small $B\backslash B^*$ };
		\node[roundnode](u) at (0,3.7)[label=north:$u$]{};
		\node[roundnode] (u2) at (-2,1.2) [label=south:$u_2$]{};
		\node[roundnode] (u1) at (-0.5, 1.2) [label=south:$u_1$]{};
		
		\node[roundnode](v) at (2,3.7)[label=north:$v$]{};
		\node[roundnode] (v1) at (2.5,1.2) [label=south:$v_1$]{};
		\node[roundnode] (v2) at (4, 1.2) [label=south:$v_2$]{};
		\foreach \from/\to in {u/u1, u2/u1,u/v,v/v1,v/v2,u1/v}
		\draw[black,very thick]  (\from) -- (\to);
		\end{tikzpicture}
		\caption*{(A3)+(B1)}
	\end{minipage}
	\begin{minipage}{0.32\linewidth}
		\centering
		\begin{tikzpicture}
		[scale=0.6, auto=left,   roundnode/.style={circle, draw, fill=black ,inner sep=0pt, minimum width=4pt}]
		\draw[gray, thick] (-3,0) rectangle (5,2.5) ;
		\draw[gray, thick] (-3,2.5) rectangle (5,5);
		\draw[gray, thick] (1,0)--(1,5);
		\filldraw[black] (-3,4.5) circle (0pt) node[anchor=west]  {\small $A^*$ };
		\filldraw[black] (-3,2 ) circle (0pt) node[anchor=west]  {\small $A\backslash A^*$ };
		\filldraw[black] (3.9,4.5) circle (0pt) node[anchor=west]  {\small $B^*$ };
		\filldraw[black] (3.3,2 ) circle (0pt) node[anchor=west]  {\small $B\backslash B^*$ };
		\node[roundnode](u) at (0,3.7)[label=north:$u$]{};
		\node[roundnode] (u2) at (-2,1.2) [label=south:$u_2$]{};
		\node[roundnode] (u1) at (-0.5, 1.2) [label=south:$u_1$]{};
		
		\node[roundnode](v) at (2,3.2)[label=right:$v$]{};
		\node[roundnode](v') at (2.5,4.2)[label=right:$v'$]{};
		\node[roundnode] (v1) at (2.5,1.2) [label=south:$v_1$]{};
		\node[roundnode] (v2) at (3.8, 1.2) [label=south:$v_2$]{};
		\foreach \from/\to in {u/u1, u2/u1,v/v1,v'/v2,u/v,u/v',u1/v}
		\draw[black,very thick] (\from) -- (\to);
		\end{tikzpicture}
		\caption*{(A3)+(B2)}
	\end{minipage}
	\begin{minipage}{0.32\linewidth}
		\centering
		\begin{tikzpicture}
		[scale=0.6, auto=left,   roundnode/.style={circle, draw, fill=black ,inner sep=0pt, minimum width=4pt}]
		\draw[gray, thick] (-3,0) rectangle (5,2.5) ;
		\draw[gray, thick] (-3,2.5) rectangle (5,5);
		\draw[gray, thick] (1,0)--(1,5);
		\filldraw[black] (-3,4.6) circle (0pt) node[anchor=west]  {\small $A^*$ };
		\filldraw[black] (-3.1,2.1 ) circle (0pt) node[anchor=west]  {\small $A\backslash A^*$ };
		\filldraw[black] (4,4.6) circle (0pt) node[anchor=west]  {\small $B^*$ };
		\filldraw[black] (3.3,2.1 ) circle (0pt) node[anchor=west]  {\small $B\backslash B^*$ };
		\node[roundnode](u) at (0,3.7)[label=north:$u$]{};
		\node[roundnode] (u2) at (-2,1.2) [label=south:$u_2$]{};
		\node[roundnode] (u1) at (-0.5, 1.2) [label=south:$u_1$]{};
		
		\node[roundnode](v) at (2,3.7)[label=north:$v$]{};
		\node[roundnode] (v1) at (2.5,1.2) [label=south:$v_1$]{};
		\node[roundnode] (v2) at (4, 1.2) [label=south:$v_2$]{};
		\foreach \from/\to in {u/u1, u2/u1,v/v1,v/v2,v1/v2,u/v,u1/v}
		\draw[black,very thick]  (\from) -- (\to);
		\end{tikzpicture}
		\caption*{(A3)+(B4)}
	\end{minipage}
    \caption{Case 3}
\end{figure}

Suppose (B1) occurs. So there exists a $b$-vertex $v_2$ in $B$ adjacent to $v$ such that $v_2\neq v_1$.
Let $(A_1,B_1)$ be obtained from $(A,B)$ by exchanging $u$ and $v_1$.
One can easily see that $u_1v_1, u_2v_1, u_2u, v_2u\notin E(G)$ (as otherwise there is a $C_4$).
So $v_1,u_1$ are $(a-1)$-vertices in $A_1$, $u_2$ is an $a$-vertex in $A_1$, $v$ is a $(b-1)$-vertex in $B_1$,
and $u$ is a $b$-vertex in $B_1$.
It is worth noting that $v_2$ may be a $(b-1)$-vertex or $b$-vertex in $B_1$,
depending on whether $v_1v_2\in E(G)$ or not, respectively.
We claim that $(A_1,B_1)\in \mathscr{P}$.
Indeed, $A_1$ is $(a-1)$-degenerate and $B_1$ is $(b-1)$-degenerate (via the same argument as in the second paragraph of Case 2).
By Claim 3, $w(A_1,B_1)=w(A,B)$, which shows $(A_1,B_1)\in \mathscr{P}$.
If $v_2$ is a $(b-1)$-vertex in $B_1$, then by Claim 4, $u_1,v_1, v, v_2$ form a $C_4$.
Otherwise $v_2$ is a $b$-vertex in $B_1$. Then we get a special path $u_2\sim u_1\sim v\sim v_2$,
implying that $u_1v_2$ or $u_2v\in E(G)$. In either case, there is a $C_4$.
This shows that (B1) cannot occur.

Suppose (B2) occurs. Then there exist a $(b-1)$-vertex $v'$ and two $b$-vertices $v_1, v_2$ in $B$
such that $v'\neq v$ and $v_1v, v_2v'\in E(G)$.
Let $(A_2,B_2)$ be obtained from $(A,B)$ by exchanging $u_1$ and $v'$.
It is easy to see that $u_1v', u_2v', vv', v_1v', v_1u_1, v_2u_1\notin E(G)$.
So $u, u_2$ are $(a-1)$-vertices in $A_2$, $v'$ is an $a$-vertex in $A_2$, $u_1, v_2$ are $(b-1)$-vertices in $B_2$,
and $v,v_1$ are $b$-vertices in $B_2$.
By Claim 3, we have $w(A_2,B_2)=w(A,B)$.
We also see that $B_2$ is $(b-1)$-degenerate (as any $b$-good subset of $B_2$ must contain $u_1$ but $u_1$ is a $(b-1)$-vertex in $B_2$),
and $A_2$ is $(a-1)$-degenerate (because any $a$-good subset of $A_2$ must contain $v'$ and all neighbors of $v'$ in $A_2$,
but $u$, as a neighbor of $v'$, is an $(a-1)$-vertex in $A_2$, a contradiction).
Therefore, $(A_2,B_2)\in \mathscr{P}$. Then by Claim 4, $u, u_2, u_1, v_2$ form a $C_4$.
This shows that (B2) cannot occur.

By the footnote of Case 2, we have seen that (B3) cannot occur.

Suppose (B4) occurs. There is a $(b+1)$-vertex $v_2$ in $B$ such that $v_2v, v_2v_1\in E(G)$.
Let $(A_4,B_4)$ be obtained from $(A,B)$ by exchanging $u$ and $v$ and exchanging $u_1$ and $v_1$.
Since $uv_1, u_1v_1, u_2u, u_2v, u_2v_1, uv_2, u_1v_2 \notin E(G)$,
we see that $v, u_2$ are $(a-1)$-vertices in $A_4$, $v_1$ is an $a$-vertex in $A_4$, $u$ is a $b$-vertex in $B_4$,
and $u_1, v_2$ are $(b-1)$-vertices in $B_4$.
By applying Claim 3 twice, we have $w(A_4,B_4)=w(A,B)$.
We claim that $(A_4,B_4)\in \mathscr{P}$.
We first show that $A_4$ is $(a-1)$-degenerate.
Suppose not, then $A_4$ has an $a$-good subset $A'$ which contains at least one of the new vertices $v, v_1$.
Since $d_{A_4}(v)=a(v)-1$, this implies that $v_1\in A'$ and moreover all neighbors of $v_1$ in $A_4$ are in $A'$,
but this is a contradiction as $v\sim v_1$.
Similarly, one can show that $B_4$ is $(b-1)$-degenerate.
Thus $(A_4,B_4)\in \mathscr{P}$. By Claim 4, $v, u_2, u_1, v_2$ form a $C_4$.
Therefore, (B4) cannot occur.

\begin{figure}[ht!]
	\begin{minipage}{0.5\linewidth}
		\centering
		\begin{tikzpicture}
		[scale=0.8, auto=left,   roundnode/.style={circle, draw, fill=black ,inner sep=0pt, minimum width=4pt}]
		\draw[gray, thick] (-3,0) rectangle (5,2.5) ;
		\draw[gray, thick] (-3,2.5) rectangle (5,5);
		\draw[gray, thick] (1,0)--(1,5);
		\filldraw[black] (-3,4.6) circle (0pt) node[anchor=west]  { $A^*$ };
		\filldraw[black] (-3.1,2.1 ) circle (0pt) node[anchor=west]  { $A\backslash A^*$ };
		\filldraw[black] (4,4.6) circle (0pt) node[anchor=west]  { $B^*$ };
		\filldraw[black] (3.6,2.1 ) circle (0pt) node[anchor=west]  { $B\backslash B^*$ };
		\node[roundnode](u) at (0,3.7)[label=north:$u$]{};
		\node[roundnode] (u2) at (-2,1.2) [label=south:$u_2$]{};
		\node[roundnode] (u1) at (-0.5, 1.2) [label=south:$u_1$]{};
		
		\node[roundnode](v) at (2,3.2)[label=right:$v$]{};
		\node[roundnode](v') at (2.5,4.2)[label=right:$v'$]{};
		\node[roundnode] (v1) at (2.5,1.2) [label=south:$v_1$]{};
		\node[roundnode] (v2) at (4, 1.2) [label=south:$v_2$]{};
		\foreach \from/\to in {u/u1, u2/u1,v/v1,v'/v2,v1/v2,u/v,u/v',u1/v}
		\draw[black,very thick] (\from) -- (\to);
		\end{tikzpicture}
		\caption*{$(A,B)$}
\medskip
	\end{minipage}
		\begin{minipage}{0.5\linewidth}
			\centering
			\begin{tikzpicture}
			[scale=0.8, auto=left,   roundnode/.style={circle, draw, fill=black ,inner sep=0pt, minimum width=4pt}]
			\draw[gray, thick] (-3,0) rectangle (5,2.5) ;
			\draw[gray, thick] (-3,2.5) rectangle (5,5);
			\draw[gray, thick] (1,0)--(1,5);
			\filldraw[black] (-3,4.6) circle (0pt) node[anchor=west]  { $A^*_5$ };
			\filldraw[black] (-3.1,2.1 ) circle (0pt) node[anchor=west]  { $A_5\backslash A^*_5$ };
			\filldraw[black] (4,4.6) circle (0pt) node[anchor=west]  { $B^*_5$ };
			\filldraw[black] (3.4,1.2 ) circle (0pt) node[anchor=west]  { $B_5\backslash B^*_5$ };
			\node[roundnode](u2) at (-0.5,4.2)[label=north:$u_2$]{};
			\node[roundnode](v) at (-0.5,3.2)[label=left:$v$]{};
			\node[roundnode] (v1) at (-0.5, 1) [label=south:$v_1$]{};
			
			\node[roundnode] (u1) at (2,3.2) [label=right:$u_1$]{};
			\node[roundnode](u) at (2,1.7)[label=south:$u$]{};

			\node[roundnode](v') at (3,1.7)[label=north:$v'$]{};
			
			\node[roundnode] (v2) at (3, 1) [label=south:$v_2$]{};
			\foreach \from/\to in {u/u1, u2/u1,v/v1,v'/v2,v1/v2,u/v,u/v',u1/v}
			\draw[black,very thick] (\from) -- (\to);
			\end{tikzpicture}
			\caption*{$(A_5,B_5)$}
\medskip
		\end{minipage}
	\begin{minipage}{0.5\linewidth}
		\centering
		\begin{tikzpicture}
		[scale=0.8, auto=left,   roundnode/.style={circle, draw, fill=black ,inner sep=0pt, minimum width=4pt}]
		\draw[gray, thick] (-3,0) rectangle (5,2.5) ;
		\draw[gray, thick] (-3,2.5) rectangle (5,5);
		\draw[gray, thick] (1,0)--(1,5);
		\filldraw[black] (-3,4.6) circle (0pt) node[anchor=west]  {\large $A^*_5$ };
		\filldraw[black] (-3.1,2.1 ) circle (0pt) node[anchor=west]  {\large $A_5\backslash A^*_5$ };
		\filldraw[black] (4,4.6) circle (0pt) node[anchor=west]  { $B^*_5$ };
		\filldraw[black] (3.4,1.2 ) circle (0pt) node[anchor=west]  { $B_5\backslash B^*_5$ };
		\node[roundnode](v'') at (-1,4)[label=north:$v''$]{};
		\node[roundnode](v) at (0,3.2)[label=left:$v$]{};
		\node[roundnode] (v3) at (-1 ,1) [label=south:$v_3$]{};
		\node[roundnode] (v1) at (0, 1) [label=south:$v_1$]{};	
			
		\node[roundnode] (u1) at (2,3.2) [label=right:$u_1$]{};
		\node[roundnode](u) at (2,1.7)[label=south:$u$]{}; 	
		\node[roundnode](v') at (3,1.7)[label=north:$v'$]{};  	
		\node[roundnode] (v2) at (3, 1) [label=south:$v_2$]{};
		\foreach \from/\to in {u/u1,v''/u1,v''/v3,v3/v1,v/v1,v'/v2,v1/v2,u/v,u/v',u1/v}
		\draw[black,very thick] (\from) -- (\to);
		\end{tikzpicture}
		\caption*{$(A_5,B_5)$}
	\end{minipage}
	\begin{minipage}{0.5\linewidth}
		\centering
		\begin{tikzpicture}
		[scale=0.8, auto=left,   roundnode/.style={circle, draw, fill=black ,inner sep=0pt, minimum width=4pt}]
		\draw[gray, thick] (-3,0) rectangle (5,2.5) ;
		\draw[gray, thick] (-3,2.5) rectangle (5,5);
		\draw[gray, thick] (1,0)--(1,5);
		\filldraw[black] (-3,4.6) circle (0pt) node[anchor=west]  {\large $A^*_6$ };
		\filldraw[black] (-3.1,2.1 ) circle (0pt) node[anchor=west]  {\large $A_6\backslash A^*_6$ };
		\filldraw[black] (4,4.6) circle (0pt) node[anchor=west]  { $B^*_6$ };
		\filldraw[black] (3.4,1.2 ) circle (0pt) node[anchor=west]  { $B_6\backslash B^*_6$ };
		\node[roundnode] (v2) at (-1,4) [label=north:$v_2$]{};
		\node[roundnode](v) at (0,3)[label=north:$v$]{};
		
		\node[roundnode](v') at (3,4)[label=north:$v'$]{};
		
		\node[roundnode] (v1) at (-1, 1) [label=south:$v_1$]{};
		\node[roundnode] (v3) at (0 , 1) [label=south:$v_3$]{};
		
		\node[roundnode](v'') at (2,1)[label=south:$v''$]{};
		\node[roundnode] (u1) at (3,1) [label=south:$u_1$]{};
		\node[roundnode](u) at (3,2)[label=right:$u$]{};
		\foreach \from/\to in {u/u1,v''/u1,v''/v3,v3/v1,v/v1,v'/v2,v1/v2,u/v,u/v',u1/v}
		\draw[black,very thick] (\from) -- (\to);
		\end{tikzpicture}
		\caption*{$(A_6,B_6)$}
	\end{minipage}
	\caption{(A3)+(B5)}
\end{figure}

Finally we assume that (B5) occurs.
Then there exist a $(b-1)$-vertex $v'$, a $b$-vertex $v_1$,
and a $(b+1)$-vertex $v_2$ in $B$ such that $v_1\sim v_2, v_1\sim v$ and $v_2\sim v'$.
Let $(A_5,B_5)$ be obtained from $(A,B)$ by exchanging $u$ and $v$ and exchanging $u_1$ and $v_1$.
As $u_2u, u_2v, u_2v_1, v'u_1, v'v, v'v_1, v_2u, v_2u_1, v_2v \notin E(G)$,
we find that $v, u_2$ is an $(a-1)$-vertex in $A_5$, $v_1$ is an $a$-vertex in $A_5$,
$u_1$ is a $(b-1)$-vertex in $B_5$, and $u, v', v_2$ are $b$-vertices in $B_5$.
And by Claim 3,  $w(A_5,B_5)=w(A,B)$.
Similarly as before, one can show that $(A_5, B_5)\in \mathscr{P}$.

Let us observe that $u_1, u, v'$ give a configuration (B3) in $B_5$.
As $uv\in E(G)$ (where $v\in A_5^*$), by the above proof of Case 3,
(A1)-(A4) cannot occur in $A_5$
(by the symmetry between the functions $a$ and $b$, here we may view $B_5, A_5$ as the new parts $A, B$, respectively).
So the configuration (A5) must occur in $A_5$.
Following the proof of Claim 7, we show that the vertices $v, v_1$ must be involved in this configuration.
Indeed, if $A_5^\diamond$ has at least two vertices, then (A1) or (A2) occurs, a contradiction.
Thus $A^\diamond$ has exactly one vertex, that is $v_1$.
Then $v_1$ has a neighbour say $v_3$ in $A_5\backslash (A_5^*\cup \{v_1\})$
such that $d_{A_5\backslash A_5^*}(v_3)= a(v_3)$.
Since neither (A3) nor (A4) occur in $A_5$,
$v_3$ must have a neighbour, say $v''$, in $A_5^*$ which is distinct from $v$.
Note that $v_3$ is an $(a+1)$-vertex in $A_5$ (see Figure 3).

Let $(A_6, B_6)$ be obtained from $(A_5, B_5)$ by exchanging $v_2$ and $v''$.
Since there is no $C_4$ in $G$, we see that $v_2v'', v_2u, v_2u_1, v_2v, v''u, v''v, v''v',v''v_1\notin E(G)$.
So $v'$ is a $(b-1)$-vertex in $B_6$, $u,u_1,v''$ are $b$-vertices in $B_6$,
$v, v_2$ are $(a-1)$-vertices in $A_6$, and $v_1$ is an $(a+1)$-vertex in $A_6$.
Clearly $A_6$ is $(a-1)$-degenerate.
If $B_6$ contains a $b$-good subset $S$, then $v''\in S$ and thus $u_1, u, v'\in S$,
contradicting that $v'$ is a $(b-1)$-vertex in $B_6$.
So $B_6$ is $(b-1)$-degenerate.
This, together with $w(A_6, B_6)=w(A_5, B_5)$ (by Claim 3), shows that $(A_6, B_6)\in \mathscr{P}$.
Then by Claim 4, $v'v_2\in E(G)$ and thus $v',v,u_1,u$ form a $C_4$.
This contradiction completes the proof of Theorem \ref{Thm: main}.
\qed

\bibliographystyle{unsrt}

\end{document}